\documentclass[11pt,reqno]{amsart}
\usepackage[margin=1in]{geometry}
\usepackage{amscd,amsfonts,amsmath,amssymb,amsthm,latexsym,mathrsfs,textcomp,verbatim}
\usepackage{accents}
\usepackage{bm}
\usepackage{bbm}
\usepackage{booktabs}
\usepackage{cite}
\usepackage{color}
\usepackage{constants}
\usepackage{csquotes}
\usepackage{enumerate}
\usepackage{etoolbox}
\usepackage[T1]{fontenc}
\usepackage{float}
\usepackage{hypbmsec}
\usepackage{mathtools}
\usepackage[dvipsnames]{xcolor}
\usepackage[breaklinks=true,colorlinks=true,linkcolor=black!25!RoyalBlue,citecolor=black!25!RoyalBlue,filecolor=RoyalBlue,urlcolor=teal]{hyperref}
\hypersetup{linktocpage}
\usepackage{soul}
\usepackage{tikz}
\usepackage{tikz-cd}
\usetikzlibrary{shapes.geometric}
\usetikzlibrary{shapes.misc}
\usetikzlibrary{positioning}
\allowdisplaybreaks[4]
\usepackage{tikz-cd}
\usepackage{marginnote}
\usepackage[all]{xy} 
\usepackage{dsfont}
\usepackage{tikz-cd}
\usepackage{stmaryrd} 
\usepackage[mathscr]{eucal}
\usepackage{hypbmsec}
\theoremstyle{plain}
\newtheorem{thm}{Theorem}[section]
\newtheorem*{thm*}{Theorem}
\newtheorem{lem}[thm]{Lemma}
\newtheorem{cor}[thm]{Corollary}
\newtheorem*{cor*}{Corollary}
\newtheorem{prop}[thm]{Proposition}

\theoremstyle{remark}
\newtheorem{rem}[thm]{Remark}
\newtheorem*{rem*}{Remark}

\theoremstyle{definition}

\newtheorem{defin}[thm]{Definition}

\newcommand{\ZZ}{\mathbb{Z}}

\newcommand{\Frob}{\mathrm{Frob}}

\newcommand{\GL}{{\operatorname{GL}}}

\newcommand{\Tor}{{\operatorname{Tor}}}

\newcommand{\Gal}{{\operatorname{Gal}}}

\DeclareMathOperator{\Res}{Res}
\DeclareMathOperator{\Sym}{Sym}
\DeclareMathOperator{\ord}{ord}

\DeclareMathOperator{\Iw}{Iw}
\DeclareMathOperator{\univ}{univ}

\DeclareMathOperator{\im}{im}

\usepackage{comment}

\title[]{Control Theorems for Hilbert Modular Varieties}

\author{Arshay Sheth}

\email{arshay.sheth@warwick.ac.uk}

\address{Mathematics Institute, Zeeman Building, University of Warwick, Coventry, CV4 7AL, UK}
\thanks{The author is supported by funding from the European Research Council under the European Union’s Horizon 2020 research and innovation programme (Grant agreement No. 101001051 — Shimura varieties and the Birch--Swinnerton-Dyer conjecture).}
\date{}

\hypersetup{
pdfauthor = {A. Sheth},
pdftitle = {Control Theorems for Hilbert Modular Varieties}
}

\begin{document}

\begin{abstract}
We prove an exact control theorem, in the sense of Hida theory, for the ordinary part of the middle degree  \'etale cohomology of certain Hilbert modular varieties, after localizing at a suitable maximal ideal of the Hecke algebra. Our method of proof builds upon the techniques introduced by Loeffler--Rockwood--Zerbes \cite{lrz}; another important ingredient in our proof is the recent work of Caraiani--Tamiozzo \cite{CT21} on the vanishing of the \'etale cohomology of Hilbert modular varieties with torsion coefficients  outside the middle degree. This work will be used in forthcoming work of the author to show that the Asai--Flach Euler system corresponding to a quadratic Hilbert modular form varies in Hida families.
\end{abstract}
\maketitle

\section[Introduction]{Introduction} 

Let $p$ be an odd prime and let $N$ be a positive integer coprime to $p$.  A fundamental theme in Hida theory is to consider the tower of modular curves 
$$
 \cdots \rightarrow Y_1(Np^r) \rightarrow \cdots  \rightarrow Y_1(Np) 
$$
corresponding to the chain of congruence subgroups
$$
\cdots \subseteq \Gamma_1(Np^r) \subseteq \cdots \subseteq \Gamma_1(Np). 
$$
The \'etale cohomology groups of this tower are packaged into the following inverse limit
$$
 H^1(Np^{\infty}) :=\varprojlim_{r} H^1_{\textrm{\'et} } (Y_1(Np^r)_{\overline{\mathbb Q}}, \mathbb Z_p(1) ),
$$
where the transition maps are taken to be the corresponding pushfoward maps in \'etale cohomology. The module $H^1(N p^{\infty})$ is equipped with an action of the adjoint Hecke operators $T'_{\ell}$ for $\ell \nmid N$ as well as the adjoint Atkin operator $U_p'$ (the usual Hecke operators $T_\ell$ and $U_p$ do not commute with the pushforward maps and hence do not act on the inverse limit).  Analogous to the usual Hida projector, we may define the adjoint or ``anti-ordinary" Hida projector by $e'_{\ord}= \lim \limits_{n \to \infty} (U_p')^{n!}$. Building on the theory of Hida on $p$-adic families of modular forms, Ohta \cite{ohta99} proved a control theorem for the anti-ordinary part of $H^1(N p^{\infty})$.  To state the theorem, we note that $H^1(N p^{\infty})$ is a module over the Iwasawa algebra $\Lambda:=\mathbb Z_p[[1+p \mathbb Z_p]] \cong \mathbb Z_p[[X]]$ via the diamond operators. Let $\Sigma$ denote the set of primes dividing $Np$ and let $G_{\mathbb Q, \Sigma}$ denote the Galois group of the maximal extension of $\mathbb Q$ unramified outside the primes in $\Sigma$. 

\begin{thm}[Ohta] \label{Ohta}
The following hold.  
\vspace{1mm}
\\
(a) We have that $e'_{\ord} H^1(N p^{\infty})$ is finite and free as a $\Lambda$-module.
\vspace{1mm}
\\
(b) For $r \geq 1$ and $k \geq 0$, let $\mathfrak p_{r, k}$ denote the ideal of $\Lambda$ generated by $(1+X)^{p^{r-1}}-(1+p)^{k p^{r-1}}$.  Then there is a canonical isomorphism 
$$
e'_{\ord} H^1(Np^{\infty} )/\mathfrak p_{r, k} \cong e'_{\ord}H^1_{\textrm{\'et} }(Y_1(Np^r)_{\overline{\mathbb Q}}, \Sym ^k (\mathbb Z_p)(1))
$$
of $\mathbb Z_p$-modules that is compatible with the action of $G_{\mathbb Q, \Sigma}$ and the Hecke operators. 
\end{thm}

A remarkable aspect of the above theorem is that the module $H^1(Np^{\infty})$, which is built only from \'etale cohomology groups with constant coefficients, also embodies information about \'etale cohomology with non-constant coefficients.  Indeed, Theorem \ref{Ohta} can be thought of as a cohomological version of the landmark work of Hida \cite{Hid86},  where he constructed a space of Lambda-adic modular forms which $p$-adically interpolate, as the weights vary, the ordinary parts of spaces of classical modular forms. Ohta's control theorem has since been used in a wide variety of different contexts. For instance, building on the ideas introduced in \cite{ohta99}, Ohta gave a new and streamlined proof of the theorem of Mazur-Wiles \cite{MW84} (the Iwasawa main conjecture over $\mathbb Q$) in a subsequent paper \cite{ohta00}. Ohta's control theorem has also been used as a crucial input by Lei-Loeffler-Zerbes \cite{llz14}
and Kings-Loeffler-Zerbes \cite{klz2} to show that the Beilinson--Flach Euler system associated to the tensor product of two modular forms varies in Hida families. 
\vspace{1mm}
\\
The main goal of this paper is to prove an analogous control theorem for the (anti)-ordinary part of certain Hilbert modular varieties after localizing at a suitable maximal ideal of the Hecke algebra. 

\subsection{Main results.} Let $F$ be a totally real number field of degree $g=[F:\mathbb Q]$ with ring of integers $\mathcal O_F$ and discriminant $\Delta_F$.  We fix an odd prime $p$ which is unramified in $F$.  Let $G=\textrm{Res}_{F/ \mathbb Q} \hspace{1mm}  \GL_2$ and let $E/\mathbb Q_p$ be a finite extension with ring of integers $\mathcal O$ such that the maximal torus of diagonal matrices splits over $E$.  
For each $n \geq 1$, we let 
$$
U_{n, p}:=\left\{ \begin{pmatrix} 
 a &  b \\
 c & d
\end{pmatrix} \in \GL_2(\mathcal O_F \otimes \mathbb Z_p):  \begin{pmatrix} 
 a &  b \\
 c & d
\end{pmatrix} \equiv \begin{pmatrix} 
x &  * \\
 0 & x
\end{pmatrix} \textrm{ mod } p^n  \textrm{ for some } x \in \mathcal O_F \otimes \mathbb Z_p \right \}. 
$$
 When $n=0$, we let $U_{n, p}$ denote the Iwahori subgroup \textit{i.e.} the subgroup of $\GL_2(\mathcal O_F \otimes \mathbb Z_p)$ consisting of those matrices which are upper triangular mod $p$. 
Let $\mathfrak N$ be an ideal of $\mathcal O_F$ which does not divide 2, 3, or $\Delta_F$. We fix a prime-to-$p$ open compact subgroup $U^{(p)} =\{ g \in G(\mathbb A_f^{ (p)}): g \equiv  \begin{pmatrix} 
* &  * \\
 0 & 1
\end{pmatrix}  \textrm{ mod } \mathfrak N \}$ and for all $n \geq 0$, we set $U_n=U^{(p)} U_{n, p}$ to be an open compact subgroup of $G(\mathbb A_f)$.  We denote by $Y_G(U_n)$ the corresponding Hilbert modular variety. The reason for working with this particular level group is explained in Remark \ref{SV5}; briefly, since our reductive group $G$ does not satisfy the SV5 axiom for Shimura varieties, we need to work with level groups having fixed and sufficiently large intersection with the center $Z_G$.

As above, we wish to study the \'etale cohomology of the Hilbert modular varieties in the tower 
$$
 \cdots \rightarrow Y_G(U_r) \rightarrow \cdots  \rightarrow Y_G(U_1)
$$
which we similarly package in the Iwasawa cohomology module 
$$H^g_{\textrm{Iw}}(Y_G(U_\infty)_{\overline{\mathbb Q}}, \mathcal O):=\varprojlim_{n} H^g_{\textrm{\'et}} (Y_G(U_n)_{\overline{\mathbb Q}}, \mathcal O),$$ where the transition maps are taken to be the corresponding pushforward maps in \'etale cohomology. 
\vspace{2mm}
\\
Let $\Sigma$ be the set of places of $F$ containing all primes dividing $\mathfrak N$ and all primes above $p$. 
Let $\mathbb T$ denote the spherical Hecke algebra generated by the standard Hecke operators $\mathcal T_v$ and $\mathcal S_v^{\pm 1}$ for $v \not \in \Sigma$ (henceforth, we use calligraphic font for our Hecke operators to avoid confusion with various level groups appearing in the paper). We let $e'_\textrm{ord}$ denote Hida's anti-ordinary projection operator;  as above, $e'_{\ord}= \lim \limits_{n \to \infty} (\mathcal U_p')^{n!}$, where $\mathcal U_p'$ is the Hecke operator corresponding to the double coset $U_{n, p} \begin{pmatrix}
    1 & 0 \\ 0 & p
    \end{pmatrix} U_{n, p}$ (see Section \ref{3.1}) . 
Both $\mathbb T$ and $e'_{\textrm{ord}}$ act on the \'etale cohomology of $Y_G(U_n)$. Let $\mathfrak m$ be a maximal ideal in the support of $H^g_{\textrm{\'et}}(Y_G(U_n)_{\overline {\mathbb Q}} , \mathbb F_p)$ such that the image of the associated Galois representation $\overline{\rho_\mathfrak m}$ is not solvable (see Section \ref{Section4}).
\vspace{2mm}
\\
Let $\Lambda=\mathcal O \llbracket (\mathcal O_F \otimes \mathbb Z_p)^{\times} \rrbracket$ \textit{i.e.} $\Lambda$ is the Iwasawa algebra over $\mathcal O$ corresponding to $(\mathcal O_F \otimes \mathbb Z_p)^{\times}$ and for each $n \geq$ 1, let $\Lambda_n$ denote the Iwasawa algebra over $\mathcal O$ corresponding to elements of $\mathcal (\mathcal O_F \otimes \mathbb Z_p)^{\times}$ which are congruent to 1 mod $p^n$. Let $\lambda$ denote a weight of $G$ that is self-dual in the sense of Section \ref{dual}; in other words $\lambda$ is a character of the maximal torus of diagonal matrices of $G$ that is trivial on the subgroup of scalar matrices. Let $V_\lambda$ denote the irreducible representation of $G$ of highest weight $\lambda$. Let $\mathcal O[-\lambda]$ denote $\mathcal O$ with $\Lambda$ acting via the inverse of the character $\lambda$.

\vspace{2mm}

\begin{thm*}[Theorem \ref{mainthm}] \label{main1} With the notation as above, the following hold: 
\vspace{2mm}
\\
(a) For all $n \geq 1$, we have that $e'_{\ord} H^g_{\Iw} (Y_G(U_\infty)_{\overline {\mathbb Q}}, \mathcal O)_{\mathfrak m}$ is free as a $\Lambda_n$-module.
\vspace{2mm}
\\
(b) For all $n \geq 1$, we have an isomorphism of $\mathcal O$-modules
$$
e'_{\ord} H^g_{\Iw}(Y_G(U_\infty)_{\overline {\mathbb Q}}, \mathcal O)_{\mathfrak m} \otimes_{\Lambda_n} \mathcal O[-\lambda] \cong e'_{\ord}H^g_{\textrm{\'et}}(Y_G(U_n)_{\overline{\mathbb Q}}, V_\lambda)_{\mathfrak m}
$$
that is compatible with the action of $G_{\mathbb Q, \Sigma}$ and $\mathbb T$. 
\vspace{2mm}
\\
(c) When $n=0$, we have a similar isomorphism
$$
e'_{\ord} H^g_{\Iw}(Y_G(U_\infty)_{\overline {\mathbb Q}}, \mathcal O)_{\mathfrak m} \otimes_{\Lambda} \mathcal O[-\lambda] \cong e'_{\ord}H^g_{\textrm{\'et}}(Y_G(U_0)_{\overline{\mathbb Q}}, V_\lambda)_{\mathfrak m}. 
$$
\end{thm*}

This theorem can be regarded as a generalisation to \'etale cohomology of Hilbert varieties of Ohta's work on \'etale cohomology of modular curves. In \cite{lrz}, Loeffler--Rockwood--Zerbes further generalize Ohta's result to Shimura varieties associated to arbitrary reductive groups. However, their results do not apply in this setting since they assume the SV5 axiom for Shimura varieties, which does not hold for our reductive group $G=\Res_{F/\mathbb Q} \GL_2$. The techniques used in this paper build on the methods introduced in \cite{lrz}, but we face additional technical difficulties due to the lack of the SV5 axiom in our setting. Throughout this paper, we have tried to emphasize the places where the SV5 axiom was needed in \textit{op.cit}, and the alternative arguments we make in the absence of this axiom. A key ingredient to obtaining  the perfect control (after localizing at $\mathfrak m$) in our main theorem is the recent work of Caraiani--Tamiozzo \cite{CT21} where they show that the \'etale cohomology of Hilbert modular varieties with torsion coefficients is concentrated in the middle degree after localizing at suitable maximal ideal of the Hecke algebra.  As a corollary of our main theorem, we can extend the vanishing results of Caraiani--Tamiozzo to  \'etale cohomology of the Hilbert varieties $Y_G(U_n)$ with non-trivial coefficients.

\begin{cor*}[Corollary \ref{lastcor}]
For all $n \geq 0$, we have $H^i_{\textrm{\'et}}(Y_G(U_n)_{\overline{\mathbb Q}}, V_{\lambda} )_{\mathfrak m} = 0$ when $i \neq g$. 
\end{cor*}

\subsection{Arithmetic applications.}
The results of the paper will be used in forthcoming work of the author \cite{She24} to show that the Asai--Flach Euler system associated to a quadratic Hilbert modular form, constructed by Lei--Loeffler--Zerbes in \cite{LLZ18}, varies in Hida families. This in turn is an important ingredient in recent work of Grossi--Loeffler--Zerbes \cite{glz2} on the proof of the Bloch--Kato conjecture in analytic rank zero for the Asai representation of a quadratic Hilbert modular form. We also expect that this work can find applications in the study of $p$-adic families of various other global cohomology classes in the Hilbert setting such as, for instance, the Hirzebruch--Zagier cycles considered in \cite{BCF19} and \cite{FJ24}.  

\subsection{Comparison with other work.} We note that there is related work of Dimitrov \cite{Dim13} which also establishes control theorems for certain Hilbert modular varieties (see Section 3 of \textit{op.cit.}), but the results in \textit{op.cit.} make stronger hypotheses on the relevant Galois representations in consideration.  In particular, the results in \cite{Dim13} are conditional on two hypotheses (a certain global big image assumption and a Fontaine--Laffaille type assumption on local weights) stated in Section 0.3 of \textit{op.cit.}; while we assume that the image $\overline{\rho_{\mathfrak m}}$ is not solvable to prove our main theorem,  we do not make any assumption similar to the second hypothesis referenced above.

\subsection{Acknowledgements}
I would like to thank my advisor David Loeffler for his support and guidance, as well as very helpful feedback on previous drafts of this paper.  I am also grateful to Rob Rockwood, Matteo Tamiozzo, Chris Williams and Ju-Feng Wu for helpful discussions in connection with this paper. Lastly, I am very grateful to the anonymous referee for their careful reading of the paper and for very helpful comments and suggestions.

\section{Background on Hilbert modular varieties} \label{Section2}

In this section, we establish some basic properties of the Hilbert modular varieties $Y_G(U_n)$ which we work with. We also follow the method in \cite{Urb11} to establish a relation between the Betti cohomology of these varieties and the group cohomology of their corresponding arithmetic subgroups. 

\subsection{Notation.} 
We begin by setting some notation that will remain fixed in the paper. 
Let $F$ be a totally real number field of degree $g$ with ring of integers $\mathcal O_F$ and discriminant $\Delta_F$. 
We fix a numbering $\{\sigma_1, \ldots, \sigma_g\}$ of real embeddings of $F$ into $\mathbb C$. We let $F^{\times +}$ (resp. $\mathcal O_F^{\times +})$ denote the totally positive elements in $F^{\times}$ (resp. $\mathcal O_F^{\times})$. Let $\mathcal H$ denote the upper half plane and let $\mathcal H_F$ denote the set of elements of $F \otimes \mathbb C$ of totally positive imaginary part; note that $\mathcal H_F$ can be identified with the product of $g$ copies of $\mathcal H$.  We let $p$ be an odd prime that is unramified in $F$.  We let $\mathbb A_f$ denote the finite adeles of $\mathbb Q$, $\mathbb A_f^{(p)}$ the finite adeles away from $p$ and $\mathbb A_{F, f}$ the finite adeles of $F$. We let $G$ to be the algebraic group $\Res_{F/\mathbb Q} \GL_2$ over $\mathbb Q$.

\subsection{Shimura varieties for $G$}
If $K \subseteq G(\mathbb A_f)$ is an open compact subgroup,  its corresponding Shimura variety $Y_G(K)$ is a quasi-projective variety with a canonical model over the reflex field $\mathbb Q$ whose complex points are given by 
$$
Y_G(K)(\mathbb C)= G(\mathbb Q)^+ \char`\\[ G(\mathbb A_f) \times \mathcal H_F]/K. 
$$
\
The Shimura varieties $Y_G(K)$ are called Hilbert modular varieties. 

\begin{defin}
We say than an open compact subgroup $K \subseteq G(\mathbb A_f)$ is sufficiently small if for every $h \in G(\mathbb A_f)$ the quotient
$$
\frac{G(\mathbb Q)^+ \cap hKh^{-1}}{K \cap \left\{ \begin{pmatrix} 
 u &  0 \\
 0 & u
\end{pmatrix}: u \in \mathcal O_F^{\times} \right \} }
$$
acts without fixed points on $\mathcal H_F$. 
\end{defin}

\begin{rem}
The above definition is slightly different from  \cite[Definition 2.2.1]{LLZ18}; we have used $\mathcal O_F^{\times}$ in the denominator rather than $\mathcal O_F^{\times +}$ used in $\textit{op.cit.}$ 
\end{rem}

If $K \subseteq  G(\mathbb A_f)$ is sufficiently small, then $Y_G(K)$ is smooth. We also note that 
if $K_1 \subseteq K_2$ is an inclusion of open compact subgroups and $K_2$ is sufficiently small, then the map $Y_G(K_1) \rightarrow Y_G(K_2)$ is a finite \'etale Galois cover. 

\subsection{The Hilbert modular variety $Y_G(U_n)$}

\begin{defin}
For each $n \geq 1$, we let 
$$
U_{n, p}:=\left\{ \begin{pmatrix} 
 a &  b \\
 c & d
\end{pmatrix} \in \GL_2(\mathcal O_F \otimes \mathbb Z_p):  \begin{pmatrix} 
 a &  b \\
 c & d
\end{pmatrix} \equiv \begin{pmatrix} 
x &  * \\
 0 & x
\end{pmatrix} \textrm{ mod } p^n  \textrm{ for some } x \in \mathcal O_F \otimes \mathbb Z_p \right \} 
$$
\end{defin}

 When $n=0$, we let $U_{n, p}$ denote the Iwahori subgroup \textit{i.e.} the subgroup of $\GL_2(\mathcal O_F \otimes \mathbb Z_p)$ consisting of those matrices which are upper triangular mod $p$. 
Let $\mathfrak N$ be an ideal of $\mathcal O_F$ which does not divide 2, 3, or $\Delta_F$. We fix a prime-to-$p$ open compact subgroup $$U^{(p)} =\left \{ g \in G(\mathbb A_f^{ (p)}): g \equiv  \begin{pmatrix} 
* &  * \\
 0 & 1
\end{pmatrix}  \textrm{ mod } \mathfrak N \right \}$$  
and let $U_n=U^{(p)} U_{n, p}$. 
By  \cite[Lemma 2.1]{Dim09}, we have that $U_n$ is sufficiently small for all $n \geq 1$ and that the determinant map $\det: U_n \rightarrow  (\mathcal O_F \otimes \widehat{\mathbb Z})^{\times}$ is surjective away from $p$. 

The inclusion $U_{n+1} \hookrightarrow U_n$ induces a map of Shimura varieties $\phi_n: Y_G(U_{n+1}) \rightarrow Y_G(U_n).$

\begin{prop} \label{deg}
The degree of the map $\phi_n$ is $[U_{n}:U_{n+1}]$. 
\end{prop}

\begin{proof}
Pick any element $[(x, g)] \in Y_G(U_n)(\mathbb C)=G(\mathbb Q)^+ \char`\\[ G(\mathbb A_f) \times \mathcal H_F]/U_n$ and note that
$$
\phi_n^{-1}( [(x, g)] ) = \{ [(x, g u_i)]: i \in I\}, 
$$
where $\{u_i\}_{i \in I}$ is a set of representatives of $U_n/U_{n+1}$.  It suffices to prove that 
$
[(x, gu_i)] \neq [(x, g u_j)]
$
when $i \neq j$. Suppose for contradiction that $[(x, gu_i)] = [(x, g u_j)]
$ when $i \neq j$.  Then there exist $h \in G(\mathbb Q)^+$ and $k \in U_{n+1}$ such that 
$(hx, hgu_ik)=(x, gu_j)$.  In particular, we conclude that $h \in G^+(\mathbb Q) \cap g U_n g^{-1}$. Since $hx=x$ and since $U_n$ is sufficiently small, we have by definition that 
$$
h \in U_n \cap \left\{ \begin{pmatrix} 
 u &  0 \\
 0 & u
\end{pmatrix}: u \in \mathcal O_F^{\times} \right\}= \left\{ \begin{pmatrix} 
 u &  0 \\
 0 & u
\end{pmatrix}: u \in \mathcal O_F^{\times} \text  { and } u \equiv 1 \mod \mathfrak N \right\}. 
$$
Thus, $h \in U_{n+1}$ as well.  Using the equality $hgu_ik= g u_j$ and the fact that $h$ lies in the centre of $G(\mathbb A_f)$, we can now conclude that $u_i h k=u_j$. Since $hk$ lies in $U_{n+1}$, this contradicts the fact that 
$\{u_i\}_{i \in I}$ is a set of representatives for $U_n/U_{n+1}$. 
\end{proof}

\begin{rem} \label{SV5}
The SV5 axiom for Shimura varieties (see \cite[pp.75]{Milne}) states that if $(G, X)$ is a Shimura datum, then the centre $Z$ is isogenous to the product of a $\mathbb Q$-split torus and an $\mathbb R$-anisotropic torus.  An equivalent formlaution is that $Z(\mathbb Q)$ is discrete in $Z(\mathbb A_f)$. 
In \cite{loeffler22}, Loeffler showed that if the SV5 axiom is satisfied, and if $K_1$ and $K_2$ are open compact subgroups of $G(\mathbb A_f)$ with $K_1 \subseteq K_2$, then the degree of the corresponding map of Shimura varieties $Y_G(K_1) \rightarrow Y_G(K_2)$ equals the index $[K_2:K_1]$. The group $G=\textrm{Res}_{F/\mathbb Q}  \GL_2$ that we are working with does not satisfy the SV5 axiom  (this is essentially because the unit group $\mathcal O_F^{\times}$ is infinite). Nevertheless, with our choice of level groups $U_n$,  Proposition \ref{deg} shows that the desired claim still holds. 
\end{rem}

\subsection{The number of components of $Y_G(U_n)$}

Let $I(F)$ denote the group of fractional ideals of $F$, and let $\textrm{Cl}^{+}(F):= I(F)/\{ (\beta): \beta  \textrm{ totally positive} \}$ denote the narrow class group of $F$. Let $h^+$ denote the narrow class number of $F$.

\begin{prop} \label{calc}
The Hilbert modular variety $Y_G(U_n)$ has $h^+ \cdot |\mathcal O_F^{\times +}\backslash (\mathcal O_F \otimes \mathbb Z_p)^{\times}  /\det(U_{n, p}) |$ connected components. 
\end{prop}

\begin{proof}
Note that the map
$\displaystyle{
\mathbb A_{F, f}^{\times} \rightarrow \textrm{Cl}^+(F) }$ defined via $\displaystyle{ (\alpha)_\mathfrak p \mapsto \prod_\mathfrak p \mathfrak p^{v_\mathfrak p(\alpha_\mathfrak p) }}$
has kernel $\prod_{\mathfrak p} \mathcal O_{F, \mathfrak p}^{\times} \cdot F^{\times +}$.   For each $x \in \textrm{Cl}^+(F)$, choose a preimage $\alpha_x \in \mathbb A_{F, f}^{\times}$. Hence, we have a decomposition $\displaystyle{\mathbb A_{F, f}^{\times}=\bigsqcup_{x \in \textrm{Cl}^+(F)} (\alpha_x \cdot \prod_{\mathfrak p} \mathcal O_{F, \mathfrak p}^{\times} \cdot F^{\times +})}$ as sets.  By strong approximation, the connected components of $Y_G(U_n)$ are indexed by $F^{\times}_+\backslash \mathbb A_{F, f}^{\times}/\det(U_n)$. This set is in bijection with
\begin{align*}
 \bigsqcup_{x \in \textrm{Cl}^+(F)} F^{\times +}\backslash \prod_{\mathfrak p} (\alpha_x \cdot  \mathcal O_{F, \mathfrak p}^{\times} \cdot F^{\times +} )/\det(U_n) 
 &\cong  \bigsqcup_{x \in \textrm{Cl}^+(F)}  \mathcal O_F^{\times^+} \backslash (\alpha_x \cdot  \prod_{\mathfrak p} \mathcal O_{F, \mathfrak p}^{\times} ) /\det(U_n) \\
& \cong  \bigsqcup_{x \in \textrm{Cl}^+(F)}  \mathcal O_F^{\times+} \backslash (\mathcal O_F \otimes \mathbb Z_p)^{\times}  /\det(U_{n, p}),
\end{align*}
where the last equality follows since the determinant map $\det: U_n \rightarrow (\mathcal O_F \otimes \widehat{\mathbb Z})^{\times}$ is surjective away from $p$. 
\end{proof}

\subsection{Hecke action} \label{hecke}

For $K \subseteq G(\mathbb A_f)$ an open compact subgroup, we let $\mathbb T_K(G)=\mathbb Z[K\backslash G(\mathbb A_f)/K]$ be the Hecke-algebra of compactly supported bi-invariant functions on $G(\mathbb A_f)$ with multiplication given by convolution.  Let $g \in G(\mathbb A_f)$ and let $K_g=K \cap g K g^{-1}$; we have a correspondence $[KgK]$

\begin{center}
\begin{tikzcd}
Y_G(K_g)\arrow[r,shift left=1ex]{}{f}\ar{d} 
  & Y_G(K_{g^{-1}}) \ar{d}
\\
Y_G(K)\arrow[r,shift left=1ex,dashed]{}{[KgK]} 
  & Y_G(K)
\end{tikzcd}
\end{center}

where the vertical maps are canonical projections and the upper-half horizontal map $f$ on complex points is induced by multiplication by $g$. We obtain an action of $\mathbb T_K(G)$ on $H^i_{\textrm{\'et}}(Y_G(K)_{\overline{\mathbb Q}}, \mathbb F_p)$ and $H^i_{\textrm{\'et}}(Y_G(K)_{\overline{\mathbb Q}}, \mathbb Z_p)$. Let $\Sigma$ be the set of places of $F$ containing 
all primes dividing $\mathfrak N$ and all primes above $p$. 
Let 
$$
\mathbb T:=\bigotimes^{'}_{v \not \in \Sigma} \mathbb Z[ \GL_2(\mathcal O_v) \backslash \GL_2(F_v)/ \GL_2(\mathcal O_v) ]
$$
denote the abstract spherical algebra away from $\Sigma$. We note that $\mathbb T$ is a subalgebra of $\mathbb T_K(G)$; it is also commutative, generated by the following Hecke operators $\mathcal T_v$ and $\mathcal S_v^{\pm 1}$ for every finite place $v \not \in \Sigma$; for  every such $v$ we choose a uniformizer $\mathcal \varpi_v$ of $\mathcal O_v$ and define 

\begin{itemize}
\item  $\mathcal T_v$ to be the double coset $\GL_2(\mathcal O_v) \begin{pmatrix}
    \varpi_v & 0 \\ 0 & 1
    \end{pmatrix} \GL_2(\mathcal O_v).$
 \item   $\mathcal S_v$ to the double coset $\GL_2(\mathcal O_v) \begin{pmatrix}
    \varpi_v & 0 \\ 0 & \varpi_v
    \end{pmatrix} \GL_2(\mathcal O_v).$
 \end{itemize}   

\subsection{Betti cohomology} \label{betti}

Let $\Gamma$ be a subgroup of $G^+(\mathbb Q)$ and let $\overline{\Gamma}$ be its image in $G(\mathbb Q)/Z_G(\mathbb Q)$. 
We assume that $\overline{\Gamma}$ has no non-trivial elements of finite order; hence, it acts freely and continuously on $\mathcal H_F$. 
In this subsection, we closely follow \cite{Urb11} to establish a relation between the group cohomology of $\overline{\Gamma}$ and the Betti cohomology of the corresponding Hilbert modular variety. 

By the work of Borel--Serre, there exists a canonical compactification $\overline{\Gamma} \backslash \overline{\mathcal H_F}$, where $\overline{\mathcal H_F}$ is a contractible real manifold with corners.  Since $\overline{\Gamma} \backslash \overline{\mathcal H_F}$ is compact, we may choose a finite triangulation of $\overline{\Gamma} \backslash \overline{\mathcal H_F}$. We may pull it back to $\overline{\mathcal H_F}$ via the canonical projection 
$ \overline{\mathcal H_F} \rightarrow \overline{\Gamma} \backslash \overline{\mathcal H_F}$. Let $C_q(\overline \Gamma)$ be the free $\mathbb Z$-module over the set of $q$-dimensional simplices of the triangulation obtained by pull-back to $\overline{\mathcal H_F}$. Since the action of $\overline{\Gamma}$ on $\overline{\mathcal H_F}$ is free,  and since the triangulation of $\overline{\Gamma} \backslash \overline{\mathcal H_F}$ is finite, the $C_q(\overline \Gamma)$'s are free $\mathbb Z[\overline{\Gamma}]$-modules of finite type. We also note that 
$$
0 \rightarrow C_d(\overline \Gamma) \xrightarrow{\partial_{d-1}} \cdots \xrightarrow{\partial_1} C_0(\overline \Gamma) \rightarrow 0
$$
is a complex computing the homology of $\overline{\mathcal H_F}$. Since $\overline{\mathcal H_F}$ is contractible, this complex is exact except in degree zero and $H_0(\overline{\mathcal H_F}, \mathbb Z)= C_0(\overline \Gamma)/\partial_0(C_1(\overline \Gamma))=\mathbb Z$.  Thus, in summary, we have that 
$$
0 \rightarrow C_d(\overline \Gamma) \xrightarrow{\partial_{d-1}} \cdots \xrightarrow{\partial_1} C_0(\overline \Gamma) \rightarrow  \mathbb Z \rightarrow 0
$$
is an exact sequence of finite free $\mathbb Z[\overline{\Gamma}]$-modules. 
If $M$ is a $\Gamma$-module,  we let
$\mathcal C^{\bullet}(\overline{\Gamma}, M)$ denote the complex 
$$
0 \rightarrow \textrm{Hom} (C_0(\overline{\Gamma}), M) \rightarrow \cdots \rightarrow  \textrm{Hom}(C_d(\overline{\Gamma}), M). 
$$
Thus,  $H^i(\overline{\Gamma}, M)$ is the $i$-th cohomology group of the complex $\mathcal C^{\bullet}(\overline {\Gamma}, M)$. 
\vspace{2mm}
\\

Let $K$ be an open compact subgroup of $G(\mathbb A_f)$ which is sufficently small and let $M$ be a left $K$-module acting via its projection to $K_p$, the image of $K$ in $G(\mathbb Q_p)$.  The corresponding Hilbert modular variety 
$Y_G(K)$ satisfies 
$$
Y_G(K) \cong  \bigsqcup_{j \in J} \Gamma_j \backslash \mathcal H_F,
$$
where $J$ is a finite set and for each $j \in J$, $\Gamma_j=g_j K g_j^{-1} \cap G^+(\mathbb Q)$ for some $g_j \in G(\mathbb A_f)$. As before, we let $\overline{\Gamma_j}$ denote the image of $\Gamma_j$ in $G(\mathbb Q)/Z_G(\mathbb Q)$.  We set 
\begin{align} \label{p}
C^{\bullet}(K, M):= \bigoplus_{j \in J} C^{\bullet}(\overline{\Gamma_j}, M)
\end{align}
\vspace{2mm}
\\
Let $\overline{Y_G}:=G(\mathbb Q)^+ \backslash G(\mathbb A_f) \times \overline{\mathcal H_F}$. Then $\overline{Y_G}(K):=\overline{Y_G}/K$ is the Borel--Serre compactification of $Y_G(K)$. 
Let $\pi: \overline{Y_G} \rightarrow \overline{Y_G}(K)$ denote the canonical projection. We choose a finite triangulation of $\overline{Y_G}(K)$ and pull it back via $\pi$. Let $C^q(K)$ denote the corresponding chain complex equipped with a right action of $K$.  Then
$$
C^{\bullet}(K, M) = \textrm{Hom}_K(C^q(K), M).   
$$
Thus, $C^{\bullet}(K, M)$ also computes the cohomology of the local system $M$ on $Y_G(K)$ and so we have an isomorphism 
\begin{align} \label{eq}
H^i(Y_G(K), M) \cong \bigoplus_{j \in J} H^i(\overline{\Gamma}_j, M). 
\end{align}

\section{Construction of a Tor spectral sequence} \label{Section3}

In this section, we construct a Tor descent spectral sequence which will be an important tool to relate the Iwasawa cohomology module $H^g_{\textrm{Iw}}(Y_G(U_\infty)_{\overline{\mathbb Q}}, \mathcal O)$  to the cohomology of $Y_G(U_n)$ at finite layers.

\subsection{General notation} \label{3.1}
We consider the group scheme $\textrm{GL}_2$ over $\mathbb Z_p$ and we let $B_2$, $N_2$ and $T_2$ to be the subgroups of upper-triangular, unipotent and diagonal matrices respectively. Following \cite{lrz}, we set the following notation:

\begin{itemize}
    \item  $Q= \textrm{Res}_{\mathcal O_F \otimes \mathbb Z_p/\mathbb Z_p} \hspace{1mm}  B_2$. 
     \item  $N= \textrm{Res}_{\mathcal O_F \otimes \mathbb Z_p/\mathbb Z_p} \hspace{1mm}  N_2$. 
    \vspace{2mm}
    \\
    \item $S= \textrm{Res}_{\mathcal O_F \otimes \mathbb Z_p/ \mathbb Z_p} \hspace{1mm}  T_2$. 
    \vspace{2mm}
    \\
    \item $E/\mathbb Q_p$ is a finite extension with ring of integers $\mathcal O$ such that $S$ splits over $E$.  
    \vspace{2mm}
    \\
  \item $S^{0}= \textrm{Res}_{\mathcal O_F \otimes \mathbb Z_p / \mathbb Z_p} \hspace{1mm}  \mathbb G_m$, viewed as a subgroup of $S$ via the diagonal embedding. 
  \vspace{1mm}
  \\
  \item $Q^0$ denotes the preimage of $S^0$ under the projection $Q \rightarrow S$
  \newline 
  \textit{i.e.} $Q^0(\mathbb Z_p)= \left \{ \begin{pmatrix}
    a & b \\ 0 & a
    \end{pmatrix} \in \GL_2(\mathcal O_F \otimes \mathbb Z_p)| a, b \in \mathcal O_F \otimes \mathbb Z_p \right \}$. 
  \vspace{1mm}
  \\
  \item $\mathfrak S=S(\mathbb Z_p)/S^0(\mathbb Z_p)$ identified with $(\mathcal O_F \otimes \mathbb Z_p)^{\times}$ via the short exact sequence
  $$
1 \rightarrow S^0(\mathbb Z_p) \rightarrow S(\mathbb Z_p) \rightarrow (\mathcal O_F \otimes \mathbb Z_p)^{\times} \rightarrow 1. 
  $$
  where the map 
  $$
S(\mathbb Z_p) \rightarrow (\mathcal O_F \otimes \mathbb Z_p)^{\times} 
\hspace{3mm} \text{ is given by } \begin{pmatrix}
    a & 0 \\ 0 & b
    \end{pmatrix} \mapsto ab^{-1}. 
  $$
  \vspace{1mm}
  \\
  \item $\Lambda=\mathcal O \llbracket \mathfrak S \rrbracket$. 
  \vspace{1mm}
  \\
  \item $\tau=\begin{pmatrix}
    p & 0 \\ 0 & 1
    \end{pmatrix} \in \GL_2(\mathbb Q_p \otimes \mathcal O_F)$.

\item $N_r=\tau^r N(\mathbb Z_p) \tau^{-r}=\left \{ \begin{pmatrix}
    1 & x \\ 0 & 1
    \end{pmatrix} \in \GL_2(\mathcal O_F \otimes \mathbb Z_p)| \hspace{1mm} x \equiv 0 \textrm{ mod  } p^r \right \} $.     

\item   $\overline{N_r}=\tau^{-r} \overline{N}(\mathbb Z_p) \tau^{r}=\left \{ \begin{pmatrix}
    1 & 0 \\ x & 1
    \end{pmatrix} \in \GL_2(\mathcal O_F \otimes \mathbb Z_p)| \hspace{1mm} x \equiv 0 \textrm{ mod  } p^r \right \} $

\item $L_r=\{\ell \in S(\mathbb Z_p): \ell \in S^0(\mathbb Z_p) \textrm{ mod  } p^r  \}=\left \{ \begin{pmatrix}
    a & 0 \\ 0 & d
    \end{pmatrix} \in \GL_2(\mathcal O_F \otimes \mathbb Z_p)| \hspace{1mm} a \equiv d \textrm{ mod  } p^r \right \}$.

\item $V_r=\overline{N_r} L_r N_0$. 
\end{itemize}

\begin{prop}
For every $r \geq 1$, we have that $V_r=U_{r, p}$. 
\end{prop}

\begin{proof}
This follows from a direct matrix computation. 
\end{proof}

We also set $V_0=\GL_2(\mathcal O_F \otimes \mathbb Z_p)$ and $V_0^{(s)}:= \tau^{-s}V_0\tau^{s} \cap V_0$ for $s \geq 1$; 
in other words, $V_0^{(s)}$ is the group of upper-triangular matrices $\textrm{mod } p^s:$

$$
V_0^{(s)=}\left \{ \begin{pmatrix}
    a & b \\ c & d
    \end{pmatrix} \in \GL_2(\mathcal O_F \otimes \mathbb Z_p)| a, b, c, d \in \mathcal O_F \otimes \mathbb Z_p \textrm{ and }  c \equiv 0 \textrm{ mod } p^s \right \}. 
$$
Similarly,  for $s \geq 0$ and $n \geq 0$, we set 
$$
U_{n, p}^{(s)}:= \tau^{-s}U_{n, p}\tau^{s} \cap U_{n, p}. 
$$

For our level groups $U_{n, p}$, we define the Hecke operator $\mathcal U_p'$ to be the double coset $U_{n, p} \begin{pmatrix}
    1 & 0 \\ 0 & p
    \end{pmatrix} U_{n, p}$  
and we define the analogue of Hida's anti-ordinary projection operator to be 
$$
e'_{\ord}= \lim_{n \to \infty} (\mathcal U_p')^{n!}. 
$$

The operator $e'_{\ord}$ acts on the cohomology of $Y_G(U_n)$. Let $M$ be a $\mathbb Z[U_{n}]$-module (acting via projection to $U_{n, p}$) with a compatible action of $\mathcal U_p'$ (in the sense that for any $u \in U_{n, p}^{(s)}$, the action of $\mathcal U_p'$ intertwines the action of $u$ and $\tau^{-s} u \tau^s$). For $K \subseteq G(\mathbb A_f)$ an open compact subgroup, the operator $e'_{\ord}$ also acts on the complexes $C^{\bullet}(K, M)$ introduced in Section \ref{betti} by lifting the action on cohomology (see \cite[pp.6]{lrz}).

\begin{lem} \label{hida}
Let $M$ be a  $\ZZ[U_{n}]$-module with a compatible action of $\mathcal U_p'$. The following diagram commutes on cohomology
$$
\begin{tikzcd}
\mathcal{C}^{\bullet}(U^{(p)}U_{n, p}^{ (s)}, M) \arrow[d, "(\mathcal U_p')^{s}"] \arrow[r, "\mathrm{cores}"] & \mathcal{C}^{\bullet}(U_{n}, M) \arrow[dl, "{[\tau^{s}]}_{*}"] \arrow[d, "(\mathcal U_p')^{s}"] \\
\mathcal{C}^{\bullet}(U^{(p)}U_{n,p}^{ (s)}, M) \arrow[r, "\mathrm{cores}"'] & \mathcal{C}^{\bullet}(U_{n}, M)
\end{tikzcd}
$$
\end{lem}

\begin{proof}
This follows from \cite[Lemma 2.7.4]{lrz} (the SV5 axiom is not needed in the proof of this lemma). 
\end{proof}

\begin{prop} \label{3.7}
The corestriction maps induce isomorphisms 
$$
 e'_{\ord} H^i( Y_G(U^{(p)}U_{n, p}^{ (s)}), M) \cong  e'_{\ord} H^i(Y_G(U_n), M) \label{cor}. 
$$
\end{prop}

\begin{proof}
As explained in \cite[Corollary 2.7.5]{lrz}, this follows from the previous lemma. 
\end{proof}

\subsection{Algebraic representations} \label{repns}

Let $X^{\bullet}(S)$ denote the character lattice of $S$ and  $X^\bullet_+(S)$ be the set of dominant weights. For each $\lambda \in X^\bullet_+(S)$, there is a unique isomorphism class of irreducible representations of $(\rho_\lambda, V_\lambda)$ of $G$ (over $E$) of highest weight $\lambda$. A representative of this isomorphism class can be constructed using the Borel--Weil--Bott theorem, as the space of all polynomials 
$$
\{ f \in E[G]: f(\overline{n} \ell g) = \lambda(\ell) f(g) \hspace{2mm} \forall \overline{n} \in \overline{N}, \ell \in S, g \in G \},
$$
with $G$ acting by right translation. More concretely, each such $\lambda \in X^\bullet_+(S)$ can be identified with an integer tuple $(k_1,\ldots ,k_g, t_1, \ldots ,t_g)$ such that the associated $V_{\lambda}$ is given by the representation $\textrm{Sym}^{k_i} V \otimes \det^{t_i}$ at the $i$-th embedding, where $V$ is the standard representation of $G$.  All our Hecke operators defined above also act on cohomology with the $V_\lambda$'s as coefficient systems (see \cite[Definition 2.5.1]{lrz}).

\subsection{Modules of measures} \label{dual} 
In this subsection, for brevity,  we let $U$ to be $U_{r, p}$ for some $r \geq 0$. We also fix a character $\lambda = (k_1, \ldots ,k_g, t_1, \ldots t_g)$ such that $k_1+2t_1= \cdots = k_g+2t_g=0$.

\begin{lem}
We have that $\lambda$ is trivial on the subtorus $S^0$. 
\end{lem}

\begin{proof}
This follows immediately from our normalization of the weights described in Section \ref{repns}; namely, 
\begin{equation*}
\lambda \left ( \begin{pmatrix}
    a & 0 \\ 0 & a
    \end{pmatrix} \right ) = \sigma_1(a)^{k_1+2t_1} \cdots  \sigma_g^{k_g+2t_g}(a)=1. \qedhere
\end{equation*}
\qedhere
\end{proof}

\begin{rem}
The weights we have considered above are exactly those which are self-dual \textit{i.e} those for which the dual of $V_{\lambda}$ is isomorphic to itself. 
\end{rem}

Following \cite{lrz}, we set the following notation. 
We define $U$- modules of continuous functions 
$$
C_{\lambda, U}:=\{f: U \rightarrow \mathcal O: f \textrm{ continuous}, f( \ell n g)=\lambda^{-1}(\ell) f(g)  \hspace{2mm} \forall ln \in (S \cap U) N_0, g \in U \}
$$
and 
$$
C_{\textrm{univ}}:=\{f: U \rightarrow \mathcal O: f \textrm{ continuous}, f(\ell n  g)= f(g) \hspace{2mm} \forall ln \in Q^0(\mathbb Z_p), g \in U \},
$$
with $U$ acting by right translation. We endow these spaces with an action of $\tau$ given by 
$$
\tau \cdot f(n \ell \bar{n})= f(n \ell \tau^{-1} \bar{n} \tau). 
$$

We define modules of bounded distributions
$$
D_{\lambda, U}:=\textrm{Hom}_{\textrm{cts}}(C_{\lambda, U}, \mathcal O)
$$
and 
$$
D_{\textrm{univ}}:=\textrm{Hom}_{\textrm{cts}}(C_{\textrm{univ}}, \mathcal O)
$$
which inherit actions of $V_0'$ and $\tau^{-1}$ by duality. We let $\mathfrak S_U$ denote the image of $U$ in $\mathfrak S$ (with respect to the Iwahori decomposition stated in Section \ref{3.1}); note that $\mathfrak S_{U} = \mathfrak S \cong (\mathcal O_F \otimes \mathbb Z_p)^{\times}$ when $U=U_{0, p}$ and $ \mathfrak S_{U}$ is the set of elements of $\mathcal (\mathcal O_F \otimes \mathbb Z_p)^{\times}$ which are congruent to 1 mod $p^n$ when $U=U_{r, p}$ and $r \geq 1$. 
Similarly, we write $\bar{N_U}$ for $\bar{N_r}$ when $U=U_{r, p}$ for $r \geq 0$.  
We let $\mathcal O \llbracket \mathfrak S_{U} \rrbracket$ denote the Iwasawa algebra corresponding to $\mathfrak S_U$. We have that $D_{\lambda, U}$ and $D_{\textrm{univ}}$ are modules over $\Lambda=\mathcal O\llbracket \mathfrak S \rrbracket$ (with the action given by inverse translation) and this structure is given explicitly by the isomorphisms (see \cite[pp.6]{lrz})
$$
D_{\lambda, U} \cong (\mathcal O[-\lambda] \otimes_{\mathcal O }\Lambda) \hat{\otimes}_\mathcal O \mathcal O \llbracket \bar{N_U} \rrbracket 
$$
and 
$$
D_{\textrm{univ}} \cong  \Lambda \hat{\otimes}_\mathcal O \mathcal O\llbracket \bar{N_U} \rrbracket,
$$
where $\mathcal O[-\lambda]$ denotes $\mathcal O$ regarded as an $\mathfrak S$ (and hence $\Lambda$) module via the inverse of $\lambda$.  In particular, we have that 
$$
D_{\textrm{univ}} \hat{\otimes}_{\mathcal O \llbracket \mathfrak S_U \rrbracket } \mathcal O[-\lambda] \cong D_{\lambda, U}
$$
as $\Lambda$-modules. Since a power series ring in finitely many variables over a Noetherian ring is flat (see \cite[pp.146]{Bou06}), we have that $D_{\univ}$ is flat as a $\Lambda$-module.

\begin{prop} \label{flat}
The anti-ordinary projector $ e'_{\ord}$ acts on $\mathcal C^{\bullet}(U, D_{\univ})$ such that we have a decomposition
$$
C^{\bullet}(U, D_{\univ}) = e'_{\ord} C^{\bullet}(U, D_{\univ}) \oplus  (1-e'_{\ord}) C^{\bullet}(U, D_{\univ}). 
$$
with $\mathcal U_p'$ acting invertibly on the first component and  topologically nilpotently on the second.
Moreover, the complex $e'_{\ord} C^{\bullet}(U, D_{\univ})$ consists of flat $\Lambda$-modules. 
\end{prop}

\begin{proof}
See  \cite[Proposition 2.7.2]{lrz}. 
\end{proof}

\subsection{Proof of the Tor spectral sequence}
In this subsection, we let $\lambda$ denote a weight that is self-dual in the sense of Section \ref{dual}. 
We let $s$ and $n$ be integers with $s \geq n$. Let $\Gamma_1=G(\mathbb Q)^+ \cap  U^{(p)} V_0^{ (s) }$ and $\Gamma_2= G(\mathbb Q)^+ \cap (  U^{(p)} (U_{n, p} \cap V_0^{ (s) }))$.  Let $\overline{\Gamma_1}$ and $\overline{\Gamma_2}$ denote the images of $\Gamma_1$ and $\Gamma_2$ in $G/Z_G(\mathbb Q)$. 
Write 
$$
Y_G( U^{ (p) } V_0^{(s)} )(\mathbb C) =  \bigsqcup_{j \in J_1} \Gamma_1 \backslash \mathcal H_F,
$$
and 
$$
Y_G( U^{ (p) } (V_0^{(s)} \cap U_{n, p}) )(\mathbb C) \cong  \bigsqcup_{j \in J_2} \Gamma_2 \backslash \mathcal H_F. 
$$
Each $i \in J_1$ corresponds to a matrix $g_i \in G(\mathbb A_f)$ whose determinants form a set of representatives for $F^{\times +}\backslash \mathbb A_{F, f}^{\times}/\det( U^{(p)} V_0^{(s)})$.  Similarly, each $i \in J_2$ corresponds to a matrix $g_i \in G(\mathbb A_f)$ whose determinants form a set of representatives for $F^{\times +}\backslash \mathbb A_{F, f}^{\times}/\det( U^{(p)} (U_{n, p} \cap V_0^{(s)}))$. 

Since $\det(U^{(p)}V_0^{(s)}))=(\mathcal O_F \otimes \hat{\mathbb Z})^{\times}$, we have that $|J_1|=h^+$, and by Lemmma \ref{calc}, $|J_2|=$  $h^+ \cdot |\mathcal O_F^{\times^+}\backslash (\mathcal O_F \otimes \mathbb Z_p)^{\times}  /\det(U_{n, p}) |$ (as written, Lemma \ref{calc} can only be applied when $s=n$ but the exact same proof goes through when $s \geq n$ since $\det(U_{n, p})=\det(U_{n, p} \cap V_0^{ (s) }) \textrm{ in } (\mathcal O_F \otimes \mathbb Z_p)^{\times} )$.  

\begin{prop} \label{cor}
Let $\varphi$ denote the natural map
$$
\varphi: \Gamma_1/\Gamma_2 \rightarrow V_0^{ (s) }/(U_{n, p} \cap V_0^{ (s) }) 
$$
Then $\im(\varphi)$ has index 
$ |\mathcal O_F^{\times+}\backslash (\mathcal O_F \otimes \mathbb Z_p)^{\times}  /\det(U_{n, p}) |
$
in  $V_0^{ (s) }/(U_{n, p} \cap V_0^{ (s) })$. 
\end{prop}

\begin{proof}
By strong approximation for the semisimple group $\textrm{Res}_{F/\mathbb Q}\textrm{SL}_2$, the diagram 

\[ \begin{tikzcd}
\Gamma_1/\Gamma_2 \arrow{r}{} \arrow[swap]{d}{\det} & V_0^{ (s) }/(U_{n, p} \cap V_0^{ (s) })  \arrow{d}{\det} \\%
 \mathcal O_F^{\times +}/(\mathcal O_F^{\times})^2 \arrow{r}{} & (\mathcal O_F \otimes \mathbb Z_p)^{\times}/\det(U_{n, p})
\end{tikzcd}
\]

is cartesian (see \cite[Corollary 3.3]{Hid00} for a similar argument). On the other hand,  note that $\det(U_{n, p})=(\mathcal O_F^{\times})^2$ in $(\mathcal O_F \otimes \mathbb Z_p)^{\times}$ by Hensel's lemma; thus we conclude that  $$\displaystyle{ |\textrm{coker}(\varphi)|=  \frac{ [ (\mathcal O_F \otimes \mathbb Z_p)^{\times}: \det(U_{n, p} ) ]}{[\mathcal O_F^{\times +}: (\mathcal O_F^{\times})^2]}}$$
and that the natural map
$
(\mathcal O_F \otimes \mathbb Z_p)^{\times}/\det(U_{n, p}) \rightarrow  \mathcal O_F^{\times+}\backslash (\mathcal O_F \otimes \mathbb Z_p)^{\times}  /\det(U_{n, p})$
has kernel $\mathcal O_F^{\times +}/ (\mathcal O_F^{\times})^2$. 
Thus, 
$\displaystyle{
\frac{ [ (\mathcal O_F \otimes \mathbb Z_p)^{\times}: \det(U_{n, p} ) ]}{[\mathcal O_F^{\times +}: (\mathcal O_F^{\times})^2]}= |\mathcal O_F^{\times^+}\backslash (\mathcal O_F \otimes \mathbb Z_p)^{\times}  /\det(U_{n, p}) |
}$
proving the claim. 

\end{proof}

\begin{prop} \label{cor1} 
We have a Hecke-equivariant isomorphism 
$$
\bigoplus_{j \in J_1}  H^i(\overline{\Gamma_1},  \mathcal O/(p^s)[-\lambda][V_0^{(s)}/ V_0^{(s)} \cap U_{n, p} ]  ) 
    \cong \bigoplus_{j \in J_2}  H^i(\overline{\Gamma_1},  \mathcal O/(p^s)[-\lambda][\overline{\Gamma_1}/\overline{\Gamma_{2}}] ). 
$$
\end{prop}

\begin{proof}
This follows from the previous proposition and the fact that $|J_1|=h^+ \text{ and } \newline 
|J_2|= h^+ \cdot |\mathcal O_F^{\times+}\backslash (\mathcal O_F \otimes \mathbb Z_p)^{\times}  /\det(U_{n, p}) |. 
$
\end{proof}

\begin{thm} \label{3.5}
We have an isomorphism of $\mathcal O$-modules 
$$
e'_{\ord}  H^i(Y_G(U_0), D_{\lambda, U_{n, p}})   \cong e'_{\ord} H^i(Y_G(U_n), V_{\lambda}). 
$$
\end{thm}

\begin{proof}
As explained in \cite[Proposition 2.7.7]{lrz}, we have an isomorphism
$$
e'_{\textrm{ord}} H^i(Y_G(U_0), D_{\lambda, U_{n, p}}/p^s)  \cong e'_{\textrm{ord}} H^i(Y_G(U^{(p)} V_0^{ (s) }), \mathcal O/(p^s)[-\lambda] \otimes_{\mathcal O[\mathfrak S_{U_{n,p}}]} \mathcal O/(p^s)[\mathfrak S]). 
$$

Hence, 
\begin{align*}
 H^i(Y_G( U^{ (p) } V_0^{(s)} ), \mathcal O/(p^s)[-\lambda] \otimes_{\mathcal O[\mathfrak S_{U_{n, p}}]} \mathcal O(p^n)[\mathfrak S]) & \cong  \bigoplus_{j \in J_1}  H^i(\overline{\Gamma_1},  \mathcal O/(p^s)[-\lambda] \otimes_{\mathcal O[\mathfrak S_{U_{n,p}}]} \mathcal O/(p^s)[\mathfrak S]) \\
   &  \cong \bigoplus_{j \in J_1}  H^i(\overline{\Gamma_1},  \mathcal O/(p^s)[-\lambda][\mathfrak S/\mathfrak S_{U_{n,p}}]  )\\
   & \cong \bigoplus_{j \in J_1}  H^i(\overline{\Gamma_1},  \mathcal O/(p^s)[-\lambda][V_0^{(s)}/ V_0^{(s)} \cap U_{n, p} ]  ) \\
   & \cong \bigoplus_{j \in J_2}  H^i(\overline{\Gamma_1},  \mathcal O/(p^s)[-\lambda][\overline{\Gamma_1}/\overline{\Gamma_{2}}] ) \\
   & \cong \bigoplus_{j \in J_2}  H^i(\overline{\Gamma_2},  \mathcal O/(p^s)[-\lambda] ) \\
   & \cong H^i(Y_G(U^{(p)}( V_0^{ (s) } \cap U_{n, p})),  \mathcal O/(p^s)[-\lambda] ). 
\end{align*}

Here, the first isomorphism follows from Equation \eqref{eq}, the second is a general property of tensor products, the third follows the fact that $\mathfrak S/\mathfrak S_{U_{n, p}} \cong V_0^{(s)}/ V_0^{(s)} \cap U_{n, p}$, the fourth follows from Proposition \ref{cor1} , the fifth follows from Shapiro's Lemma and the sixth follows from Equation \eqref{eq} again. As explained in \cite[Proposition 2.7.7]{lrz}, we can use Proposition \ref{3.7} to conclude that $e'_{\textrm{ord}} H^i(Y_G(U^{(p)}( V_0^{ (s) } \cap U_{n, p})),  \mathcal O/(p^s)[-\lambda] )$ is in turn isomorphic to $e'_{\textrm{ord}}  H^i(Y_G(U_n),  V_\lambda/p^s)$ .  Thus, combining our isomorphims, we conclude that 
\begin{align*}
e'_{\textrm{ord}} H^i(Y_G(U_0), D_{\lambda, U_{n, p}}) &= \varprojlim_{s} e'_{\textrm{ord}} H^i(Y_G(U_0), D_{\lambda, U_{n, p}}/p^s)  \\ &\cong  \varprojlim_{s} e'_{\textrm{ord}}  H^i(Y_G(U_n),  V_\lambda/p^s) \\
&=e'_{\textrm{ord}} H^i(Y_G(U_n) , V_{\lambda}). 
\end{align*}
\end{proof}

\begin{rem}
The analogue of this theorem in \cite{lrz} (\cite[Proposition 2.7.7]{lrz}) crucially used the SV5 axiom for Shimura varieties in the application of Shapiro's lemma.  Nevertheless, as the proof given above demonstrates, with our choice of level groups, we do not need to invoke this axiom. 
\end{rem}

We note that Theorem \ref{3.5} is compatible with the comparision isomorphism between Betti and \'etale cohomology; for the rest of the paper, we work with \'etale cohomology of Hilbert modular varieties.

\begin{cor}
We have an isomorphism 
$$
e'_{\ord} H^i_{\textrm{\'et}}(Y_G(U_0), D_{\univ})  \cong e'_{\ord}  \varprojlim_{n} H^i_{\textrm{\'et}} (Y_G(U_n), \mathcal O)
$$
\end{cor}

\begin{proof}
Setting $\lambda$ to be the trivial character we deduce from the previous theorem that 
\begin{align*}
   e'_{\textrm{ord}} H^i_{\textrm{\'et}}(Y_G(U_0), D_{\textrm{univ}}) & \cong \varprojlim_{Q^0 \subseteq U, n}  e'_{\textrm{ord}} H^i_{\textrm{\'et}}(Y_G( U_0), D_{\lambda, U}/(p^n)) \\
   & \cong   \varprojlim_{n}  e'_{\textrm{ord}}  H^i_{\textrm{\'et}}(Y_G(U_n), \mathcal O/(p^n)) \\
   & \cong e'_{\textrm{ord}}  \varprojlim_{n} H^i_{\textrm{\'et}} (Y_G(U_n), \mathcal O). 
\end{align*}
Here the last isomorphism follows from the facts that our inverse system satisfies the Mittag-Leffler property because the cohomology groups in the system are finitely generated, and and that $e'_{\ord}$ commutes with the maps in the inverse limit. 
\end{proof}

Define $M^{\bullet}$ to be the image of $e'_{\textrm{ord}} C^{\bullet}(U_0, D_{\textrm{univ}})$ in the subcategory $D^{\textrm{flat}}(R)$ of the derived category of $R$-modules generated by flat objects.  We also set $\Lambda_n= \mathcal O \llbracket \mathfrak S_{U_{n, p}} \rrbracket$ for all $n \geq 1$.

\begin{thm}
For all $n \geq 1$, we have a quasi-isomorphism 
$$
M^{\bullet} \otimes^{\mathbb L}_ {\Lambda_n} \mathcal O[-\lambda] \cong e'_{\ord} C^{\bullet}(U_n, V_\lambda). 
$$
\end{thm}

\begin{proof}
Since $M^{\bullet}$ is represented by the flat complex $e'_{\textrm{ord}} \mathcal C^\bullet(U_0, D_{\textrm{univ}})$, we can compute the derived tensor product as
$$
e'_{\textrm{ord}} \mathcal C^\bullet (U_0, D_{\textrm{univ}}) \otimes_{\Lambda_n} \mathcal O[-\lambda] = e'_{\textrm{ord}} \mathcal C^\bullet(U_0, D_{\textrm{univ}} \otimes_{\Lambda_n} \mathcal O[-\lambda]). 
$$ 
By Theorem \ref{3.5}, this complex is isomorphic to $\ e'_{\textrm{ord}} C^{\bullet}(U_n, V_\lambda)$. 
\end{proof}

\begin{cor} \label{spectral1}
For all $n \geq 1$, there is a spectral sequence 
$$
E_2^{i, j}: \Tor_{-i}^{\Lambda_n}(e'_{\ord} \varprojlim_{s} H^j_{\textrm{\'et}}  (Y_G(U_s), \mathcal O) _{\overline{\mathbb Q}}, \mathcal O[-\lambda]) \implies e'_{\ord} H^{i+j}_{\textrm{\'et}} (Y_G(U_n)_{\overline{\mathbb Q}}, V_\lambda). 
$$
\end{cor}

\begin{proof}
This follows from the previous theorem using the spectral sequence for the Tor functor. 
\end{proof}

An analogous argument to the one given in this section allows us to obtain a similar result at the Iwahori level $U_0$.

\begin{cor} \label{spectral2}
There is a spectral sequence 
$$
E_2^{i, j}: \Tor_{-i}^{\Lambda}(e'_{\ord} \varprojlim_{s} H^j_{\textrm{\'et}}  (Y_G(U_s), \mathcal O)_{\overline{\mathbb Q}}, \mathcal O[-\lambda]) \implies e'_{\ord} H_{\textrm{\'et}} ^{i+j}(Y_G(U_0)_{\overline{\mathbb Q}}, V_\lambda). 
$$
\end{cor}

\section{Proof of the control theorem} \label{Section4}

In this section, we use the spectral sequences in Corollaries \ref{spectral1} and \ref{spectral2} to give a proof of our control theorem. We begin by recalling the results of Caraiani--Tamiozzo \cite{CT21}, which play a crucial role in proving our control theorem.

\subsection{The results of Caraiani--Tamiozzo} \label{caraianitamiozzo}
Let $K \subseteq G(\mathbb A_f)$ be a neat compact open subgroup and
 take a maximal ideal $\mathfrak m \subseteq \mathbb T$ in the support of $H^i(Y_G(K)_{\overline{\mathbb Q}}, \mathbb F_p)$. By the work of Scholze \cite{Sch15}, we have a unique continuous semisimple Galois representation
$$
\overline{\rho_{\mathfrak m}}: \Gal (\overline{F}/F) \rightarrow \GL_2(\overline{\mathbb F_p})
$$
such that $\overline{\rho_{\mathfrak m}}$ is unramified for all $v$ not in a finite set of suitable places of $F$ and such that the characteristic polynomial of $\overline{\rho_{\mathfrak m}}(\Frob_v)$ equals $X^2-\mathcal T_vX+\mathcal S_v N(v) \textrm{ mod } \mathfrak m$.

\begin{thm}[Caraiani--Tamiozzo] \label{ct}
For any $i \in \mathbb N_{\geq 1}$, let $\mathfrak m$ be a maximal ideal in the support of $H^i_{\textrm{\'et}}(Y_G(K)_{\overline{\mathbb Q}}, \mathbb F_p)$. Assume that the image of $\overline{\rho_\mathfrak m}$ is not solvable. Then $H^i_{\textrm{\'et}}(Y_G(K)_{\overline{\mathbb Q}}, \mathbb F_p)_{\mathfrak m}$ is non-zero only for $i=g$. 
\end{thm}

\begin{proof}
See \cite[Theorem 7.1.1]{CT21}. 
\end{proof}

\begin{cor} \label{modpvanishing}
In the above setting, we have $H^i_{\textrm{\'et}}(Y_G(K)_{\overline{\mathbb Q}}, \mathbb Z_p)_{\mathfrak m} \neq 0$ only for $i=g$. Moreover, $H^g_{\textrm{\'et}}(Y_G(K)_{\overline{\mathbb Q}}, \mathbb Z_p)_{\mathfrak m}$ is free as a $\mathbb Z_p$-module. 
\end{cor}

\begin{proof}
This is explained in \cite[Corollary 7.1.2]{CT21}; we give the proof below for the convenience of the reader. We consider the short exact sequence
$$
0 \rightarrow \mathbb Z_p \xrightarrow{\cdot p} \mathbb Z_p \rightarrow \mathbb F_p  \rightarrow 0,  
$$
localize the corresponding long exact sequence at $\mathfrak m$ and employ Theorem \ref{ct} to obtain a surjective map 
$$
H^i_{\textrm{\'et}}(Y_G(K)_{\overline{\mathbb Q}}, \mathbb Z_p)_{\mathfrak m} \xrightarrow{\cdot p} H^i_{\textrm{\'et}}(Y_G(K)_{\overline{\mathbb Q}}, \mathbb Z_p)_{\mathfrak m}
$$
for all $i \neq g$.  Since $H^i_{\textrm{\'et}}(Y_G(K)_{\overline{\mathbb Q}}, \mathbb Z_p)_{\mathfrak m}$ is a finitely generated $\mathbb Z_p$-module, we can apply Nakayama's lemma to conclude that $H^i_{\textrm{\'et}}(Y_G(K)_{\overline{\mathbb Q}}, \mathbb Z_p)_{\mathfrak m}=0$ for all $i \neq g$. When $i=g$, we get that 
$$
\frac{H^g_{\textrm{\'et}}(Y_G(K)_{\overline{\mathbb Q}}, \mathbb F_p)_{\mathfrak m}}{\textrm{im}(H^g_{\textrm{\'et}}(Y_G(K)_{\overline{\mathbb Q}}, \mathbb Z_p)_{\mathfrak m} \rightarrow H^g_{\textrm{\'et}}(Y_G(K)_{\overline{\mathbb Q}}, \mathbb F_p)_{\mathfrak m}}=0.
$$
Since $H^g_{\textrm{\'et}}(Y_G(K)_{\overline{\mathbb Q}}, \mathbb F_p)_{\mathfrak m}  \neq 0$, it follows that $H^g_{\textrm{\'et}}(Y_G(K)_{\overline{\mathbb Q}}, \mathbb Z_p)_{\mathfrak m} \neq 0$ as well.  Finally, note that the long-exact sequence also yields an injection $H^g_{\textrm{\'et}}(Y_G(K)_{\overline{\mathbb Q}}, \mathbb Z_p)_\mathfrak m \xhookrightarrow{\cdot p} H^g_{\textrm{\'et}}(Y_G(K)_{\overline{\mathbb Q}},  \mathbb Z_p)_\mathfrak m$. Hence, 
$H^g_{\textrm{\'et}}(Y_G(K)_{\overline{\mathbb Q}}, \mathbb Z_p)_\mathfrak m$ has no $p$-torsion, hence is torsion-free and hence free (using the fact that torsion-free modules over a PID are free). 
\end{proof}

\subsection{Proof of the control theorem}

We begin by recalling some notation and assumptions.  We let $H^g_{\textrm{Iw}}(Y_G(U_\infty)_{\overline{\mathbb Q}}, \mathcal O)= \varprojlim_{n} H^g_{\textrm{\'et}} (Y_G(U_n)_{\overline{\mathbb Q}}, \mathcal O)$, $\Lambda=\llbracket(\mathcal O_F \otimes \mathbb Z_p)^{\times}\rrbracket$ and $\Lambda_n= \mathcal O \llbracket \mathfrak S_{{U_{n, p}}} \rrbracket $. We recall that $\Lambda_n$ is just the Iwasawa algebra corresponding to elements of $\mathcal (\mathcal O_F \otimes \mathbb Z_p)^{\times}$ which are congruent to 1 mod $p^n$.
Let $\mathfrak m$ be a maximal ideal in the support of $H^g_{\textrm{\'et}}(Y_G(U_n)_{\overline {\mathbb Q}} , \mathbb F_p)$ such that the image of the associated Galois representation $\overline{\rho}_\mathfrak m$ is not solvable.  Let $\lambda$ denote a weight of $G$ that is self-dual in the sense of Section \ref{dual}. Let $\mathcal O[-\lambda]$ denote $\mathcal O$ with $\Lambda$ acting via the inverse of the character $\lambda$.  The first step in the proof of the control theorem is to analyze the $E_2^{i, j}$ term in Corollary \ref{spectral1}. To do so, we first recall the following result from commutative algebra. 

\begin{prop} \label{koszul}
Let $R$ be a commutative ring, let $x_1, \ldots, x_N$ be elements of $R$ and let $R^N$ be the free $R$-module of rank $N$ with basis $\{e_i : 1 \leq i \leq N \}$. Consider the Koszul complex $K_{\bullet}(x_1, \ldots, x_N)$  associated to $x_1, \ldots, x_N$ given by 
$$
0 \rightarrow \bigwedge^{N} R^{N} \xrightarrow{d} \bigwedge^{N-1} R^{N} \xrightarrow{d} \cdots \xrightarrow{d} \bigwedge^{1} R^N \xrightarrow{d} R \rightarrow 0, 
$$
where $$
d(e_{j_1} \wedge \cdots \wedge e_{j_i} )= \sum_{s=1}^{i} (-1)^{s-1} x_{j_s} e_{j_1} \wedge \cdots \wedge \widehat{e_{j_s}} \wedge \cdots \wedge e_{j_i}.  
$$
If $x_1, \ldots ,x_N$ form a regular sequence in $R$, then $K_{\bullet}(x_1, \ldots, x_N)$ is a free-resolution of $R/(x_1, \ldots, x_N)$ as an $R$-module. 
\end{prop}

\begin{proof}
See for instance \cite[Corollary 4.5.5]{Wei94}. 
\end{proof}

\begin{prop} \label{tor}
We have that the $E_{2}^{i,j}$ term in Corollary \ref{spectral1} is zero unless $i \in \{0,-1,\ldots ,-g\}$.
\end{prop}

\begin{proof}
Since $E_2^{i, j}= \Tor_{-i}^{\Lambda_n}(e'_{\ord} \varprojlim_{s} H^j_{\textrm{\'et}}  (Y_G(U_s)_{\overline{\mathbb Q}}, \mathcal O), \mathcal O[-\lambda])$,  it suffices to construct a free resolution of $\mathcal O[-\lambda]$ as a $\Lambda_n$-module of length $g$. On the other hand, the Iwasawa algebra $\Lambda_n$ can be identified with products of copies of $\mathcal O[[T_1, \ldots, T_g]]$; it thus suffices to construct a free resolution of $\mathcal O[-\lambda]$ as an $\mathcal O[[T_1, \ldots, T_g]]$ module of length $g$. To do this, we note that the sequence $T_1-\lambda^{-1}(T_1), \ldots, T_g-\lambda^{-1}(T_g)$ is a regular sequence in $ O[[T_1, \ldots, T_g]]$ and that $$ O[[T_1, \ldots, T_g]]/(T_1-\lambda^{-1}(T_1), \ldots, T_g-\lambda^{-1}(T_g)) \cong \mathcal O[-\lambda]$$ as $\mathcal O[[T_1, \ldots, T_g]]$-modules.  We can thus apply Proposition \ref{koszul} to deduce that the Koszul complex associated to $T_1-\lambda^{-1}(T_1), \ldots, T_g-\lambda^{-1}(T_g)$ provides a free resolution of $\mathcal O[-\lambda]$ as an $\mathcal O[[T_1, \ldots, T_g]]$ module of length $g$. 
\end{proof}

\begin{rem} \label{oo}
 By using a similar argument as above, we can also deduce that $E_{2}^{i,j}$ term in Corollary \ref{spectral2} is zero unless $i \in \{0,-1,\ldots ,-g\}$.
\end{rem}

\begin{prop} \label{better} 
The following hold: 
\vspace{2mm}
\\
(a)  For all $n \geq 1$, we have that 
$$
\Tor^{\Lambda_n}_{-i}(e'_{\ord}H^g_{\Iw}(Y_G(U_\infty)_{\overline{\mathbb Q}}, \mathcal O )_\mathfrak m, \mathcal O[-\lambda])= H^{i+g}_{\textrm{\'et}} (Y_G(U_n)_{\overline{\mathbb Q}}, V_\lambda)_{\mathfrak m}. 
$$

(b) When $n=0$,  we have that 
$$\Tor^{\Lambda}_{-i}(e'_{\ord}H^g_{\Iw}(Y_G(U_\infty)_{\overline{\mathbb Q}}, \mathcal O )_\mathfrak m, \mathcal O[-\lambda])= H^{i+g}_{\textrm{\'et}} (Y_G(U_0)_{\overline{\mathbb Q}}, V_\lambda)_{\mathfrak m}. $$ 
\end{prop}

\begin{proof}
(a) Since localization is exact and commutes with the Tor functor, we deduce from Corollary \ref{spectral1} the following  spectral sequence: 
$$
F_2^{i, j}:=(E_2^{i,j})_\mathfrak m =\Tor^{\Lambda_n}_{-i}(e'_{\textrm{ord}}H^j_{\textrm{Iw}}(Y_G(U_\infty)_{\overline {\mathbb Q}}, \mathcal O)_\mathfrak m, \mathcal O[-\lambda]) \implies  H^{i+j}_{\textrm{\'et}} (Y_G(U_n), V_\lambda)_{\mathfrak m}. 
$$
By Proposition \ref{tor} and Corollary \ref{modpvanishing}, we deduce that $F_2^{i,j}=0$ unless $(i, j) \in \{ (0,g), \ldots (-g, g)\}$. It follows that the spectral sequence degenerates on the second page and thus the desired isomorphism follows directly from applying the filtration theorem. 
\vspace{2mm}
\\
(b) Using the spectral sequence in Corollary \ref{spectral2} and Remark \ref{oo}, this follows by using a similar argument as in part (a). 
\end{proof}

\begin{thm} \label{mainthm} With the notation as above, the following hold: 
\vspace{2mm}
\\
(a) For all $n \geq 1$, we have that $e'_{\textrm{ord}} H^g_{\Iw}(Y_G(U_\infty)_{\overline {\mathbb Q}}, \mathcal O)_{\mathfrak m}$ is free as a $\Lambda_n$-module.
\vspace{2mm}
\\
(b) For all $n \geq 1$, we have an isomorphism of $\mathcal O$-modules
$$
e'_{\textrm{ord}} H^g_{\Iw}(Y_G(U_\infty)_{\overline {\mathbb Q}}, \mathcal O)_{\mathfrak m} \otimes_{\Lambda_n} \mathcal O[-\lambda] \cong e'_{\textrm{ord}}H^g_{\textrm{\'et}}(Y_G(U_n)_{\overline{\mathbb Q}}, V_\lambda)_{\mathfrak m}. 
$$
that is compatible with the action of $G_{\mathbb Q, \Sigma}$ and $\mathbb T$. 
\vspace{2mm}
\\
(c) When $n=0$, we have a similar isomorphism
$$
e'_{\ord} H^g_{\Iw}(Y_G(U_\infty)_{\overline {\mathbb Q}}, \mathcal O)_{\mathfrak m} \otimes_{\Lambda} \mathcal O[-\lambda] \cong e'_{\ord}H^g_{\textrm{\'et}}(Y_G(U_0)_{\overline{\mathbb Q}}, V_\lambda)_{\mathfrak m}. 
$$
\end{thm}

\begin{proof}
We first deduce parts (b) and (c) from the results above. 
\vspace{2mm}
\\
(b) This follows by setting $i=0$ in Proposition \ref{better} (a). 
\vspace{2mm}
\\
(c) 
This follows by setting $i=0$ in Proposition \ref{better} (b). 
\vspace{2mm}
\\
We now prove part (a). 
\vspace{2mm}
\\
(a) Let $\pi$ and $k$ denote the uniformizer and residue field of $\mathcal O$ respectively. 
Consider the short-exact sequence
$$
0 \rightarrow \mathcal O \xhookrightarrow{\cdot \pi} \mathcal O \twoheadrightarrow  k  \rightarrow 0. $$  This gives rise to a long-exact sequence of Tor groups: 
\begin{align*}
& \cdots \rightarrow \Tor^{\Lambda_n}_{1}(e'_{\textrm{ord}} H^g_{\textrm{Iw}}(Y_G(U_\infty)_{\overline {\mathbb Q}}, \mathcal O)_{\mathfrak m}, \mathcal O) \xrightarrow{\cdot \pi}  \Tor^{\Lambda_n}_{1}(e'_{\textrm{ord}} H^g_{\textrm{Iw}}(Y_G(U_\infty)_{\overline {\mathbb Q}}, \mathcal O)_{\mathfrak m}, \mathcal O) \rightarrow \\
& \Tor^{\Lambda}_{1}(e'_{\textrm{ord}}H^g_{\textrm{\'et}}(Y_G(U_\infty)_{\overline {\mathbb Q}}, \mathcal O)_\mathfrak m, k ) \rightarrow e'_{\textrm{ord}} H^g_{\textrm{Iw}}(Y_G(U_\infty)_{\overline {\mathbb Q}}, \mathcal O)_{\mathfrak m} \otimes_{\Lambda_n} \mathcal O \xrightarrow{\cdot \pi} 
 H^g_{\textrm{Iw}}(Y_G(U_\infty)_{\overline {\mathbb Q}}, \mathcal O)_{\mathfrak m} \otimes_{\Lambda_n} \mathcal O  \rightarrow \cdots  
\end{align*} 
By setting $i=0$ and $\lambda$ to be the trivial character in Proposition \ref{better}, and then applying Corollary \ref{modpvanishing}, we get that 
$$
 \Tor^{\Lambda_n}_{1}(e'_{\textrm{ord}}H^g_{\textrm{Iw}}(Y_G(U_\infty)_{\overline{\mathbb Q}}, \mathcal O_v)_\mathfrak m, \mathcal O)=0. 
$$
On the other hand, Corollary \ref{modpvanishing} implies that $e'_{\textrm{ord}} H^g_{\textrm{Iw}}(Y_G(U_\infty)_{\overline {\mathbb Q}}, \mathcal O)_{\mathfrak m} \otimes_{\Lambda_n} \mathcal O$
 has no $\pi$-torsion, so the long exact sequence gives us that $\Tor^{\Lambda_n}_{1}(e'_{\textrm{ord}}H^g_{\textrm{\'et}}(Y_G(U_\infty)_{\overline {\mathbb Q}}, \mathcal O)_\mathfrak m, k )=0$. Finally, we note that part (b) and the topological Nakayama's lemma (see \cite[pp.226]{PB97}) imply that $e'_{\textrm{ord}} H^g_{\textrm{Iw}}(Y_G(U_\infty)_{\overline {\mathbb Q}}, \mathcal O)_{\mathfrak m}$ is finitely generated as a $\Lambda_n$-module; 
 we can now use the local criterion of flatness to conclude that $e'_{\textrm{ord}} H^g_{\textrm{Iw}}(Y_G(U_\infty)_{\overline {\mathbb Q}}, \mathcal O)_{\mathfrak m} $ is flat as a $\Lambda_n$-module. Since a finitely generated flat module over a noetherian local ring is free, we have that $e'_{\textrm{ord}} H^g_{\textrm{Iw}}(Y_G(U_\infty)_{\overline {\mathbb Q}}, \mathcal O)_{\mathfrak m}$ is free as a $\Lambda_n$-module. 

\begin{cor} \label{lastcor}
For all $n \geq 0$, we have $H^i_{\textrm{\'et}}(Y_G(U_n)_{\overline{\mathbb Q}}, V_{\lambda} )_{\mathfrak m} = 0$ when $i \neq g$. 
\end{cor}

\begin{proof}

When $n \geq 1$, we recall from Proposition \ref{better} (a) that $
\Tor^{\Lambda_n}_{-i}(e'_{\ord}H^g_{\Iw}(Y_G(U_\infty)_{\overline{\mathbb Q}}, \mathcal O )_\mathfrak m, \mathcal O[-\lambda])= H^{i+g}_{\textrm{\'et}} (Y_G(U_n)_{\overline{\mathbb Q}}, V_\lambda)_{\mathfrak m}. 
$
On the other hand, we know from Theorem \ref{mainthm} that $e'_{\textrm{ord}} H^g_{\Iw}(Y_G(U_\infty)_{\overline {\mathbb Q}}, \mathcal O)_{\mathfrak m}$ is free as a $\Lambda_n$-module.  Thus, the above Tor group vanishes when $i \neq 0$ and so we conclude as desired that $H^i_{\textrm{\'et}}(Y_G(U_n)_{\overline{\mathbb Q}}, V_{\lambda} )_{\mathfrak m} = 0$ when $i \neq g$.  The case $n=0$ follows similarly using Proposition \ref{better} (b). 
\end{proof}

\end{document}